%% file: project.tex
\documentclass[a4paper,12pt]{article}

\usepackage[english]{babel}
\usepackage{amssymb,amsthm, amsmath}

\usepackage{graphicx,color}
\usepackage{cleveref}
\usepackage{caption}
\usepackage[font={footnotesize}]{subcaption}

\newtheorem{thm}{Theorem}[section]
\newtheorem{lem}[thm]{Lemma}
\newtheorem{pro}[thm]{Proposition}
\theoremstyle{definition}

\crefname{thm}{Theorem}{Theorems}
\crefname{lem}{Lemma}{Lemmas}
\crefname{pro}{Proposition}{Propositions}
\crefname{figure}{Figure}{Figures}
\crefname{table}{Table}{Tables}

\numberwithin{figure}{section}
\numberwithin{table}{section}

\title{Betti numbers of cut ideals of trees}

\author{
Samu Potka and Camilo Sarmiento
}
\date{}

\begin{document}

\maketitle

\begin{abstract}
Cut ideals, introduced by Sturmfels and Sullivant, are used in phylogenetics and algebraic statistics.
We study the minimal free resolutions of cut ideals of tree graphs. By employing basic methods from topological combinatorics, we obtain upper bounds for the Betti numbers of this type of ideals. These take the form of simple formulas on the number of vertices, which arise from the enumeration of induced subgraphs of certain incomparability graphs associated to the edge sets of trees. 
\end{abstract}

\section{Introduction}

Let $G$ be a simple graph. We denote the set of its vertices by $V(G)$ and the set of its edges by $E(G)$. By a \emph{cut} $A|B$ of $G$, we mean a partition of $V(G)$ into two subsets, $A,B\subseteq V(G)$ (so that $A\cap B=\emptyset$ and $A\cup B=V(G)$). Partitions are considered unordered, hence the number of cuts of a graph $G$ is $2^{|V(G)|-1}$.
 
Note that a given cut $A|B$ also partitions  the set of edges into two subsets: $S_{A|B}\subseteq E(G)$, consisting of those edges whose endpoints lie in different parts, and $T_{A|B}\subseteq E(G)$, consisting of those edges whose endpoints lie in the same part.

We may associate a toric ideal to a graph $G$ as follows. Fix a field $\Bbbk$, and introduce the polynomial rings $R_G:=\Bbbk[r_{A|B}:A|B \text{ a cut of } G ]$ and $S_G:=\Bbbk[s_e,t_e: e\in E(G)] $, in $2^{|V(G)|-1}$ and $2|E(G)|$ indeterminates, respectively. Define the ring homomorphism:
\begin{align}
& \phi_G:R_G \to S_G \nonumber \\
 & r_{A|B} \mapsto \prod_{e\in S_{A|B}}s_e\cdot \prod_{e\in T_{A|B}}t_e \nonumber.
\end{align}
The \emph{cut ideal} of a graph $G$ is:
\begin{displaymath}
I_G:=\ker \phi_G\subset R_G.
\end{displaymath}

Cut ideals were introduced by Sturmfels and Sullivant in \cite{SS06}. They showed how to obtain generators for the cut ideal of a graph $G$ in terms of the generators of two cut ideals $I_{G_1}$ and $I_{G_2}$ in the case where $G$ is a zero-, one- or two-sum of two graphs $G_1$ and $G_2$. They also obtained necessary and sufficient conditions for the toric variety $\mathbf{V}(I_G)$ to be smooth, among other results. Later, Nagel and Petrovi\'c \cite{NP08} proved that cut ideals of ring graphs admit a quadratic Gr\"obner basis and established algebraic properties which derive from this fact, such as Cohen-Macaulayness and Koszulness of the coordinate ring $R_G/I_G$. In \cite{Eng11}, Engstr\"om proved a conjecture presented in \cite{SS06}, namely that the cut ideals of $K_4$-minor free graphs are generated by quadrics.

We shall study minimal free resolutions of cut ideals of the class of \emph{tree graphs}. A tree graph is a connected graph without cycles. From \cite{NP08}, we know that their cut ideals are arithmetically Gorenstein, and from \cite{Eng11} that they are generated by quadrics. We denote a tree on $n$ vertices by $T_n$. By applying some topological ideas introduced by Engstr\"om and Dochtermann in \cite{DE09}, we were able to obtain estimates for some of the Betti numbers of the cut ideals of $T_n$:

\begin{thm}
\label{thetheorem}
For the Betti numbers of $I_{T_{n+1}}$, we have the following bounds, which hold independently of the underlying field:
\begin{align}
\beta_{0,2}(I_{T_{n+1}}) \leq \beta_{0,2}(\mathrm{in}\,I_{T_{n+1}}) &= \frac{1}{2}\left(4^{n} - 2\! \cdot \! 3^{n} +2^{n} \right) \nonumber \\
\beta_{1,3}(I_{T_{n+1}}) \leq \beta_{1,3}(\mathrm{in}\,I_{T_{n+1}}) &=  \frac{1}{3}\left( 8^{n}-3\! \cdot \! 6^{n} +3\! \cdot \! 4^{n}-2^{n} \right)\nonumber \\
\beta_{2,4}(I_{T_{n+1}}) \leq \beta_{2,4}(\mathrm{in}\,I_{T_{n+1}}) &=  \frac{1}{8} \left(16^{n}-4\! \cdot \!12^{n}+6\! \cdot \! 8^{n}+2\! \cdot \!7^{n}-4\! \cdot \!6^{n}+4\! \cdot \!5^{n}\right. \nonumber \\
&\qquad \left. -9\! \cdot \! 4^{n}+2\! \cdot \! 3^{n}+2\! \cdot \!2^{n}\right)\nonumber \\
\beta_{1,4}(I_{T_{n+1}}) \leq \beta_{1,4}(\mathrm{in}\,I_{T_{n+1}}) &=  \frac{1}{4} \left( 7^{n}-4\! \cdot \! 6^{n}+6\! \cdot \! 5^{n}-4\! \cdot \! 4^{n}+3^{n}\right) \nonumber \\
\beta_{3,5}(I_{T_{n+1}}) \leq\ \beta_{3,5}(\mathrm{in}\,I_{T_{n+1}}) &=  \frac{1}{60} \left(2\! \cdot \! 32^{n}-10\! \cdot \! 24^{n}+30\! \cdot \! 20^{n}-120\! \cdot \! 18^{n}+30\! \cdot \! 17^{n}\right. \nonumber \\
& \qquad \left.-40\! \cdot \! 16^{n}+180\! \cdot \! 15^{n}+375\! \cdot \! 14^{n}-420\! \cdot \! 13^{n}-180\! \cdot \! 12^{n} \right. \nonumber\\
&\qquad \left. +200\! \cdot \! 11^{n}-280\! \cdot \! 10^{n}-220\! \cdot \! 9^{n}+985\! \cdot \! 8^{n}-720\! \cdot \! 7^{n}\right. \nonumber \\
& \qquad \left.+655\! \cdot \! 6^{n}-710\! \cdot \! 5^{n} +35\! \cdot \! 4^{n}+340\! \cdot \! 3^{n}-132\! \cdot \! 2^{n} \right) \nonumber \\
\beta_{2,5}(I_{T_{n+1}}) \leq \beta_{2,5}(\mathrm{in}\,I_{T_{n+1}}) &=  \frac{1}{12} \left(3\! \cdot \! 14^n -12\! \cdot \! 12^n -2\! \cdot \! 11^n + 22\! \cdot \!
10^n - 2\! \cdot \! 9^n\right. \nonumber \\
& \qquad\left.-9\! \cdot \! 8^{n} - 6\! \cdot \! 7^n + 9\! \cdot \! 6^n - 10\! \cdot \! 5^n + 11\! \cdot \!
4^{n} - 4\! \cdot \! 3^n\right) \nonumber \\
\nonumber
\end{align}
\end{thm}

We emphasize that, since the main algebraic results used to obtain this bounds, that is Hochster's formula (cf. \cref{thm:hochster} below) and upper semi-continuity of Betti numbers (cf. \cref{lem:upper}), apply for an arbitrary field, our arguments adapt almost directly to obtain bounds for the Betti numbers of cut ideals of trees over any field (see last paragraph in the proof of \cref{thetheorem}).  

\subsection{Definitions}
We recall some definitions that are necessary for presenting our proof of \cref{thetheorem}. A \emph{simplicial complex} $\Delta$ is a collection of subsets of a finite base set $[n]:=\{1,2,\ldots , n\}$ which is closed under taking subsets. The \emph{independence complex} $\mathsf{Ind}(G)$ of a graph $G$ is the simplicial complex on $V(G)$ whose faces are the subsets of $V(G)$ which are not adjacent in $G$. The \emph{Stanley-Reisner ideal} of a simplicial complex $\Delta$ on $n$ vertices is the square-free monomial ideal:
\begin{displaymath}
I_\Delta:=\langle x_{i_1}x_{i_2}\ldots x_{i_r}: (i_1,i_2,\ldots,i_r)\notin \Delta \rangle\subseteq \Bbbk[x_1,\ldots,x_n].
\end{displaymath}
One of the cornerstones of combinatorial commutative algebra is \emph{Hochster's formula}, relating the Betti numbers of the minimal free resolution of a Stanley-Reisner ideal $I_{\Delta}$, with the homology of the subcomplexes of $\Delta$: 
\begin{thm}[Hochster's formula, \cite{MS05}]\label{thm:hochster}
For $i > 0$, the Betti numbers $\beta_{i,j}$ of the Stanley-Reisner ideal of a simplicial complex $\Delta$ are given by:
\begin{equation}
\beta_{i,j}\left(I_\Delta\right) = \sum_{\substack{F\subseteq V(\Delta) \\ |F|=j}}\dim_\Bbbk \tilde{H}_{j-i-1}(\Delta[F],\Bbbk),
\label{hoch}
\end{equation}
where $\Delta[F]$ refers to the subcomplex of $\Delta$ induced by the vertices in $F$.
\end{thm}

The \emph{edge ideal} of a graph $G$ on $n$ vertices is the monomial ideal defined as $\langle x_ix_j : (ij)\in E(G) \rangle\subseteq \Bbbk[x_1,\ldots, x_n]$.

\section{Proof of the main result}

A couple of preliminary computations using \texttt{Macaulay2} provide us with the first two total Betti numbers of the cut ideals, over $\mathbb{Q}$, for some path graphs $P_n$. They are listed in \cref{tab:tab1} \footnote{The blank in the last entry means the computation had not concluded after roughly one hour. For the next heaviest tasks, the computation of the syzygies of $I_{P_7}$ and the generators of $I_{P_8}$, the CPU time was 55.7 minutes and 8.8 seconds, respectively. We used a 8 AMD Opteron Dual-Core 2.6 GHz computing server with 64 GB RAM running Ubuntu 12.04. In each case, the generators were obtained with the \emph{4ti2}\cite{4ti2} interface for \texttt{Macaulay2}.}. 
\begin{table}[htbp!]
\begin{center}\begin{tabular}{c|cc}
 & $\beta_0$ & $\beta_1$ \\
 \hline
$I_{P_3}$ & 1 & 0 \\
$I_{P_4}$ & 9 & 16 \\
$I_{P_5}$ & 55 & 320 \\
$I_{P_6}$ & 285 & 4160 \\
$I_{P_7}$ & 1351 & 44800 \\
$I_{P_8}$ & 6069 & -
\end{tabular}
\caption{First two total Betti numbers of cut ideals of path graphs.}
\label{tab:tab1}
\end{center}
\end{table} 

We observe that the number of generators and first syzygies increases quickly when considering larger paths. Accordingly, direct use of the usual functions in \texttt{Macaulay2} (for example, \texttt{gens}, \texttt{syz}, \texttt{res}) becomes unfeasible when doing the computations. This motivates the general strategy of trying to relate combinatorial properties of the graphs to algebraic properties of their cut ideals, which we shall follow now.

Our approach derives from the results of Engstr\"om and Dochtermann in \cite{DE09} and consists of two steps:
\begin{itemize} 
\item Construct a Gr\"obner basis for $I_{T_n}$, and characterize the initial ideal from which it arises combinatorially. 
\item Regard the initial ideal gotten as the Stanley-Reisner ideal of a certain simplicial complex, and use \emph{Hochster's formula} to get an estimate for the Betti numbers of $I_{T_n}$.
\end{itemize}

To have a picture of the initial ideals associated to $I_{T_n}$, we present the complete Betti diagrams for their minimal free resolutions in \cref{tab:diag} for $n=4$ and $n=5$. Here we use the standard monomial ordering from \texttt{Macaulay2}, namely graded reverse lexicographic.
\begin{table}[htbp!]
\hspace{3cm}$n=4$\begin{minipage}[c]{0.9\linewidth}
\begin{center}
\begin{verbatim}
      total: 1 9 16 9 1
          0: 1 .  . . .
          1: . 9 16 9 .
          2: . .  . . 1
\end{verbatim}
\end{center}
\end{minipage}
\\[0.8cm]
$n=5$\begin{minipage}[c]{0.9\linewidth}
\begin{verbatim}
      total: 1 55 326 951 1744 2273 2273 1744 951 326 55  1
          0: 1  .   .   .    .    .    .    .   .   .  .  .
          1: . 55 320 897 1462 1437  836  282  54   6  .  .
          2: .  .   6  54  282  836 1437 1462 897 320 55  .
          3: .  .   .   .    .    .    .    .   .   .  .  1
\end{verbatim}
\end{minipage}
\caption{Betti diagrams for initial ideals of cut ideals of trees on four and five vertices.}
\label{tab:diag}
\end{table}

We start with an observation. Note that the identification $A|B\mapsto S_{A|B}$ defines an injective mapping from the set of cuts of $G$ into $\mathbf{2}^{E(G)}$. If $G$ is a tree, this map is a surjection onto $\mathbf{2}^{E(G)}$, and hence every subset of $E(G)$ corresponds bijectively to a cut of $G$. This allows us to think of the indeterminates in $R_{G}$ equivalently as being labelled by the subsets of $E(G)$ (see \cref{fig2}).

\begin{figure}[h]
\begin{center}
\begin{subfigure}[b]{0.5\textwidth}
\centering
\includegraphics[width=\textwidth]{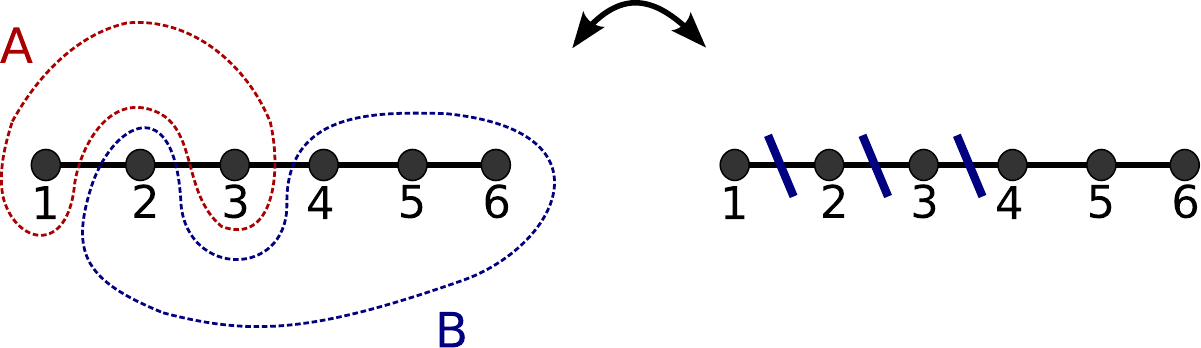}
\caption{Representation as a subset of edges.}
\label{fig2}
\end{subfigure} \qquad
\begin{subfigure}[b]{0.4\textwidth}
\centering
\includegraphics[width=\textwidth]{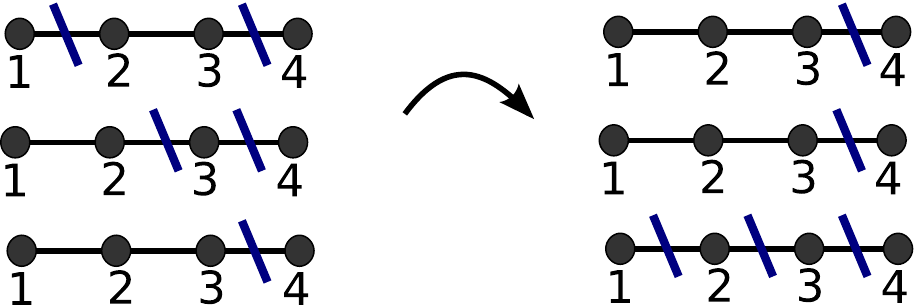}
\caption{Passing to normal form.}
\label{fig3}
\end{subfigure}
\caption{Equivalence of cuts and subsets of edges.}
\end{center}
\end{figure} 

We wish to introduce a \emph{normal form} for the monomials in $R_{T_n}$. To describe it, picture a monomial $m$ as stacked subsets of edges; then the normal form of $m$ is obtained by \emph{sending edges to the bottom} as illustrated in \cref{fig3}. As it turns out, the normal form of a monomial can be reached in steps by successively sending edges to the bottom for pairs of cuts in the monomial. The translation into algebraic terms goes as follows.

\begin{pro}
Let $I_{T_n}$ be a cut ideal associated to a tree $T_n$. Then there is a term order in $R_{T_n}$, with respect to which the set
\begin{equation*}
\mathcal{B}=\{r_X\cdot r_Y - r_{X \cup Y}\cdot r_{X \cap Y} : X,Y\subset E({T_n}) \mbox{ are incomparable }\}\subset I_{T_n}
\end{equation*}
is a Gr\"obner basis for $I_{T_n}$ (here we designate a cut of ${T_n}$ by the subset of $E({T_n})$ it cuts).
\label{gro}
\end{pro} 
\begin{proof}
(Adapted from the proof of Theorem 9.1 in \cite{S96}.) Assign a weight to the indeterminate $r_X$ as the number of elements $X' \subset E(T_n)$ incomparable with $X$. Let $\prec$ be any term order refining the partial order given by the weights to a total order. We claim that $\mathcal{B}$ is a Gr\"obner basis for $I_{T_n}$ with respect to $\prec$. Assume to the contrary, that there is at least one binomial $b=m-m'\in I_{T_n}$ such that m is not divisible by any $r_X\cdot r_Y$ with $X,Y \subset E({T_n})$ incomparable. This implies that all indeterminates appearing in $m$ are mutually comparable. 

By taking $b$ to be the minimal binomial providing a counterexample, we may suppose that $m$ and $m'$ have no common factors, and further, that the (labels of the) indeterminates in $m'$ are also mutually comparable (otherwise we may reduce $m'$ modulo $\mathcal{B}$ to have it that way). Then $m$ and $m'$ have disjoint sets of indeterminates. But since $b$ belongs to $I_{T_n}$, $m$ and $m'$ must cut the same edges the same number of times (and equally for edges kept together). This can only happen if $m=m'$, because only then can the indeterminates inside $m$ and $m'$ be mutually comparable, thus leading us to a contradiction. Hence, $m$ must be divisible by some $r_X\cdot r_Y$ with $X,Y \subset E({T_n})$ incomparable.
\end{proof}

We will use the next lemma towards the end of the proof of \cref{thetheorem}. The reader can refer to \cite{MS05} for further details.

\begin{lem}[Upper semicontinuity, Theorem 8.29 in \cite{MS05}]
Fix a graded ideal $I$ in a polynomial ring $\Bbbk [x_1,\ldots, x_n]$. If $\mathrm{in}(I)$ is the initial ideal of $I$ with respect to some term order, then:
\begin{equation*}
\beta_{i,j}\left(I\right)\leq \beta_{i,j}\left(\mathrm{in}(I)\right) \mbox{ for all }i,j\in \mathbb{N}
\end{equation*}
\label{lem:upper}
\end{lem}

\begin{proof}[Proof of \cref{thetheorem}]
Let $\mathrm{in}\, I_{{T_n}}$ be the initial ideal giving rise to the Gr\"obner basis of $I_{T_n}$ presented in \cref{gro}. This monomial ideal can be regarded as the \emph{edge ideal} of a graph $\Gamma_{E({T_n})}$ with the elements of $\mathbf{2}^{E({T_n})}$ as the vertex set and the pairs of the incomparable subsets of $E({T_n})$ as the edge set. We call $\Gamma_{E({T_n})}$ the \emph{incomparability graph of} $\mathbf{2}^{E({T_n})}$; clearly, it only depends on $|E({T_n})|=n-1$.

\begin{figure}[h]
\centering
\includegraphics[width=0.15\textwidth]{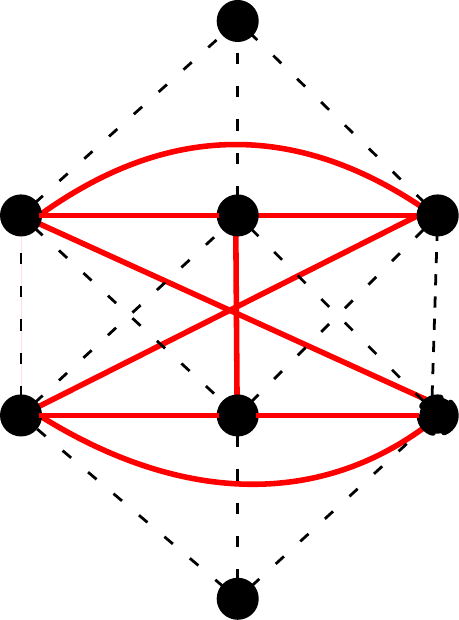}
\caption{$\Gamma_{E({T_n})}$ for $n=4$.}
\label{fig4}
\end{figure}

Now, extracting the insight from \cite{DE09}, we consider $\mathrm{in}\,I_{{T_n}}$ as the Stanley-Reisner ideal of the simplicial complex $\mathsf{Ind}(\Gamma_{E({T_n})})$. This means that we can use Hochster's formula to calculate the Betti numbers of $\mathrm{in}\,I_{{T_n}}$ with knowledge of the counts of induced subgraphs of $\Gamma_{E({T_n})}$ and the (dimension of the) reduced homology of their independence subcomplexes:
\begin{equation}
\beta_{i,j}\left(\mathrm{in}\,I_{T_n}\right)=\beta_{i,j}\left(I_{\mathsf{Ind}(\Gamma_{E({T_n})})}\right)= \sum_{\substack{F\subseteq V(\Gamma_{E({T_n})}) \\ |F|=j}}\dim_\Bbbk \left(\tilde{H}_{j-i-2}(\mathsf{Ind}(\Gamma_{E({T_n})}[F]))\right).
\label{hochstergraph}
\end{equation}

The enumeration of the induced subgraphs of $\Gamma_{E({T_n})}$ is a straightforward combinatorial calculation, which can be performed using inclusion-exclusion. We implemented this procedure in a \texttt{Python} script \cite{script} to get formulas for the number of induced subgraphs, and illustrate it with the enumeration of  \raisebox{-0.1ex}{\makebox[2ex][l]{\includegraphics[width=2ex]{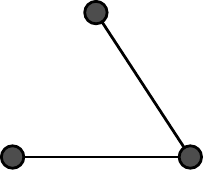}}}. 

We write the elements of $\Gamma_{E({T_n})}$ as tuples in $\{0,1\}^{n-1}$. Then, we can compare $X,Y,Z\in\Gamma_{E({T_n})} $ by indicating the number of entries $i$ for which $(X_i,Y_i,Z_i)$ attains every possible value. Let these numbers be $a,b,\ldots, h$, as in \cref{table3}.
\begin{figure}[h]
\begin{minipage}[c]{0.45\linewidth}
\begin{center}\begin{tabular}{c|cccccccc}
&$a$ & $b$ & $c$ & $d$ &$e$ &$f$ &$g$ &$h$ \\
\hline
$X$&0 & 1 & 0 & 1&0 & 1 & 0 & 1\\
$Y$&0 & 0 & 1 & 1&0 & 0 & 1 & 1\\
$Z$& 0& 0& 0&0&1&1&1&1 
\end{tabular} 
\end{center}
\end{minipage}\qquad
\begin{minipage}[c]{0.45\linewidth}
\begin{center}
\def\svgwidth{0.4\linewidth}
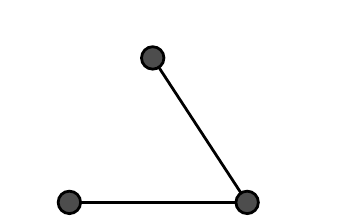
\end{center}
\end{minipage}
\caption{Comparison of $X,Y,Z\in\Gamma_{E({T_n})} $.}
\label{table3}
\end{figure}

The condition for the edge $XY$ to be present, for instance, is that both $b+f$ and $c+g$ be larger than zero. Hence, the number of labelled occurences of \raisebox{-0.1ex}{\makebox[2ex][l]{\includegraphics[width=2ex]{g33.pdf}}} in $\Gamma_{E({T_n})}$ is given by the sum:
\begin{align}
\sum_{\substack{a+b+\ldots+h=n-1, \\b+f>0,\ c+g>0,  \\c+d>0,\ e+f>0,\\b+d=0\ \mathtt{xor} \ e+g=0}}\binom{n-1}{a,b,\ldots,h}=&
\sum_{\substack{a+c+e+f+g+h=n-1,\\f>0,\ c+g>0,  \\c>0,\ e+f>0,\\b+d=0}}\binom{n-1}{a,c,e,f,g,h} \nonumber \\
&+\sum_{\substack{a+b+c+d+f+h=n-1,\\ b+f>0,\ c>0,  \\c+d>0,\ f>0,\\e+g=0}}\binom{n-1}{a,b,c,d,f,h} \nonumber \\
&-2\sum_{\substack{a+c+f+h=n-1,\\f>0,\ c>0,  \\ b+d=0, e+g=0}}\binom{n-1}{a,c,f,h}, \nonumber
\end{align}
The terms on the right hand side are decomposed according to inclusion-exclusion:
\begin{equation*}
\sum_{\substack{a+c+f+h=n-1\\f>0}}=\sum_{a+c+f+h=n-1}-\sum_{\substack{a+c+h=n-1\\f=0}},
\end{equation*}
and by taking symmetry into account, we obtain:
\begin{equation*}
\# \mbox{\raisebox{-0.1ex}{\makebox[2ex][l]{\includegraphics[width=2ex]{g33.pdf}}}}  (\Gamma_{E({T_n})})=6^{n-1}-2\cdot5^{n-1}+2\cdot 3^{n-1}-2^{n-1}.
\end{equation*}
Finally, \cref{tab:graphs} contains the (dimensions of the reduced) homologies of the independence complexes of the induced subgraphs making contributions. Note that, as long as the independence complexes of the graphs involved in equation \ref{hochstergraph} have torsion-free homology groups, we do not have to concern about the underlying field $\Bbbk$.
Putting this together with the counts for the induced subgraphs, we obtain the formulas in \cref{thetheorem}. 
The fact that these expression bound the Betti numbers of $I_{T_n}$ from above is a consequence of the well-known upper semicontinuity for the Betti numbers of a minimal free resolution, stated in \cref{lem:upper}. Thus, we conclude the proof.
\end{proof}

\begin{table}[h]\small
\begin{minipage}[c]{0.45\linewidth}
\begin{center}\begin{tabular}{|c|cc|}
\hline
\rule{0pt}{2.5ex} 
  Graph &  $\dim \tilde{H}_0$ & $\dim \tilde{H}_1$ \\
 \hline
 \rule{0pt}{2.5ex}
 \raisebox{-0.3ex}{\makebox[2ex][l]{\includegraphics[width=2ex]{g33.pdf}}} & 1 & 0  \\
 \raisebox{-0.3ex}{\makebox[2ex][l]{\includegraphics[width=2ex]{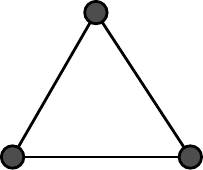}}} & 2 & 0  \\[0.05cm]
\hline
\end{tabular} 
\end{center}
\end{minipage} 
\hspace{0.5cm}
\begin{minipage}[c]{0.45\linewidth}
\begin{center}\begin{tabular}{|c|cc|}
\hline
\rule{0pt}{2.5ex} 
  Graph &  $\dim \tilde{H}_0$ & $\dim \tilde{H}_1$ \\
 \hline
 \rule{0pt}{2.5ex}
 \raisebox{-0.3ex}{\makebox[2ex][l]{\includegraphics[width=2ex]{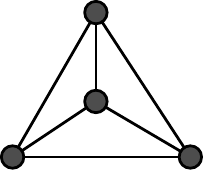}}} & 3 & 0  \\
 \raisebox{-0.3ex}{\makebox[2ex][l]{\includegraphics[width=2ex]{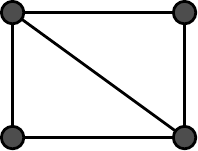}}} & 2 & 0  \\
 \raisebox{-0.3ex}{\makebox[2ex][l]{\includegraphics[width=2ex]{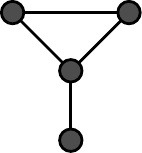}}} & 1 & 0  \\
\raisebox{-0.3ex}{\makebox[2ex][l]{\includegraphics[width=2ex]{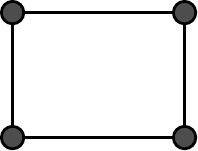}}} & 1 & 0  \\ 
\raisebox{-0.3ex}{\makebox[2ex][l]{\includegraphics[width=2ex]{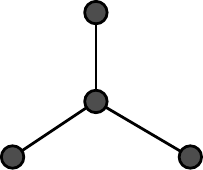}}} & 1 & 0  \\
\raisebox{-0.3ex}{\makebox[2ex][l]{\includegraphics[width=2ex]{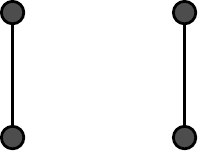}}} & 0 & 1  \\[0.05cm]
\hline
\end{tabular}
\end{center}
\end{minipage}
\begin{center}\begin{tabular}{|c|cc|c|cc|}
\hline
\rule{0pt}{2.5ex} 
  Graph &  $\dim \tilde{H}_0$ & $\dim \tilde{H}_1$ &  Graph & $\dim \tilde{H}_0$ &  $\dim \tilde{H}_1$ \\
 \hline
 \rule{0pt}{3ex}\hspace{-0.2cm}
\raisebox{-0.3ex}{\makebox[2ex][l]{\includegraphics[width=2.5ex]{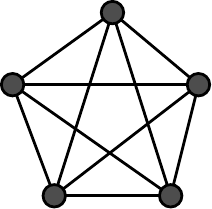}}} & 4 & 0 & \raisebox{-0.3ex}{\makebox[2ex][l]{\includegraphics[width=2.5ex]{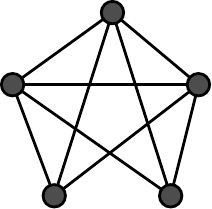}}} & 3 & 0 \\
\raisebox{-0.3ex}{\makebox[2ex][l]{\includegraphics[width=2ex]{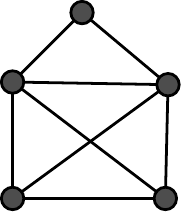}}} & 2 & 0 & \raisebox{-0.3ex}{\makebox[2ex][l]{\includegraphics[width=2ex]{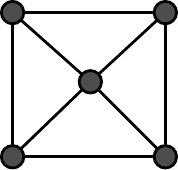}}} & 2 & 0 \\
\raisebox{-0.3ex}{\makebox[2ex][l]{\includegraphics[width=2ex]{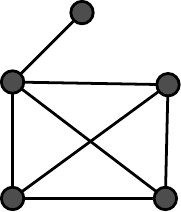}}} & 1 & 0 & \raisebox{-0.3ex}{\makebox[2ex][l]{\includegraphics[width=2ex]{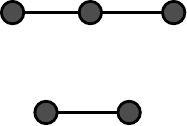}}} & 0 & 1 \\
\raisebox{-0.3ex}{\makebox[2ex][l]{\includegraphics[width=2ex]{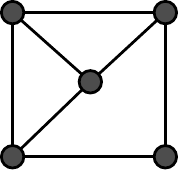}}} & 1 & 0 & \raisebox{-0.3ex}{\makebox[2ex][l]{\includegraphics[width=2ex]{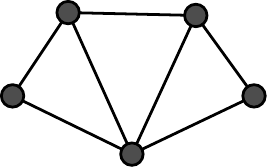}}} & 1 & 0 \\
\raisebox{-0.3ex}{\makebox[2ex][l]{\includegraphics[width=2ex]{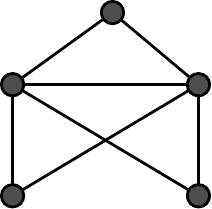}}} & 2 & 0 & \raisebox{-0.3ex}{\makebox[2ex][l]{\includegraphics[width=2ex]{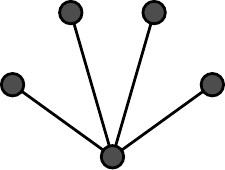}}} & 1 & 0 \\
\raisebox{-0.3ex}{\makebox[2ex][l]{\includegraphics[width=2ex]{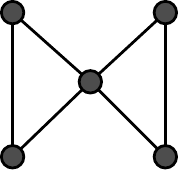}}} & 1 & 1 & \raisebox{-0.3ex}{\makebox[2ex][l]{\includegraphics[width=2ex]{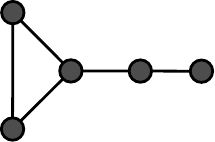}}} & 0 & 1 \\
\raisebox{-0.3ex}{\makebox[2ex][l]{\includegraphics[width=2ex]{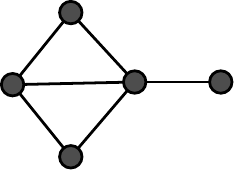}}} & 1 & 0 & \raisebox{-0.3ex}{\makebox[2ex][l]{\includegraphics[width=2ex]{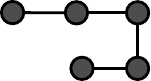}}} & 0 & 1 \\
\raisebox{-0.3ex}{\makebox[2ex][l]{\includegraphics[width=2ex]{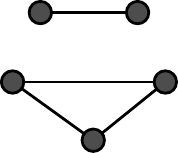}}} & 0 & 2 & \raisebox{-0.3ex}{\makebox[2ex][l]{\includegraphics[width=2ex]{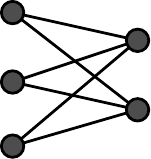}}} & 1 & 0 \\
\raisebox{-0.3ex}{\makebox[2ex][l]{\includegraphics[width=2ex]{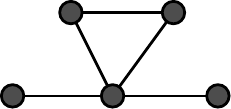}}} & 1 & 0 & \raisebox{-0.3ex}{\makebox[2ex][l]{\includegraphics[width=2ex]{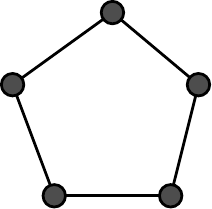}}} & 0 & 1 \\[0.05cm]
\hline
\end{tabular}
\end{center}
\caption{\small Contributions to the reduced homology over $\mathbb{Z}$ of the independence complexes of graphs. Notice that there is no torsion in any of these homology groups.}
\label{tab:graphs}
\end{table}

Below, in \cref{tab2}, are the estimates for the Betti numbers of $I_{{T_n}}$ for $n$ up to nine.

\begin{table}[htbp!]
\begin{center}\begin{tabular}{c|cccccc}
 $n$& $\beta_{0,2}$ & $\beta_{1,3}$&$\beta_{2,4}$&$\beta_{1,4}$&$\beta_{2,5}$ &$\beta_{3,5}$  \\
 \hline
$3$ & 1 & 0 &0 &0 & 0 & 0\\
$4$ & 9 & 16 &9 & 0 & 0 & 0\\
$5$ & 55 & 320 & 897&6 & 54 & 1450 \\
$6$ & 285 & 4160 & 32025&150 & 3380 & 156824\\
$7$ & 1351 & 44800 & 810255 & 2280& 115950 & 9798758\\
$8$ & 6069 & 435356 & 17298519 &27300  & 2984380 & 474814396\\
$9$ & 26335 & 3978240 & 335187657&283626 & 64924734 & 19911592842
\end{tabular} \caption{The Betti numbers of the initial ideals. Obtained by using subgraph counts and the homology of independence complexes.}
\label{tab2}
\end{center}
\end{table}

\section{Further remarks}

A few interesting questions were left unaddressed in this note. 
\begin{enumerate}
\item Does the combinatorial description we used for the initial ideal of $I_{T_n}$ provide information about the cellular complexes supporting the minimal cellular resolution of $\mathrm{in}\,I_{{T_n}}$?. Our initial computations of the minimal free resolutions of $\mathrm{in}\,I_{{T_n}}$  with $\texttt{Macaulay2}$ showed that the minimal free resolution of a cut ideal of a tree on four vertices is supported by the polytopal complex in \cref{fig:fig5}.
\begin{figure}[htbp!]
\begin{center}
\includegraphics[width=0.35\textwidth]{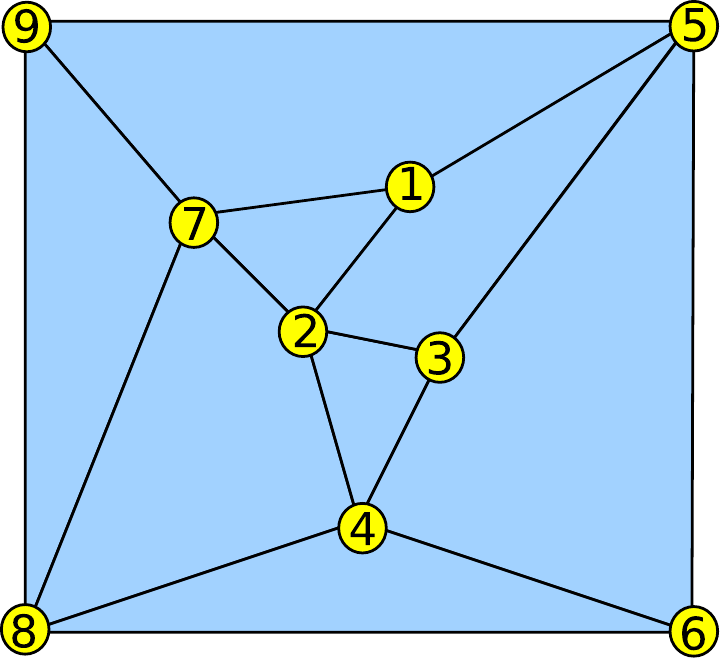}
\caption{A Schlegel diagram of a polytopal complex supporting the minimal free resolution of the initial monomial ideal of $P_4$. The monomials corresponding to the vertices are, respectively, $r_{1|234}r_{12|34},$ $r_{1|234}r_{123|4},$ $r_{12|34}r_{123|4},$ $r_{2|134}r_{123|4},$  $r_{12|34}r_{14|23},$ $r_{2|134}r_{14|23},$ $r_{1|234}r_{124|3},$ $r_{2|134}r_{124|3},$ $r_{14|23}r_{124|3}$.}
\label{fig:fig5}
\end{center}
\end{figure}

\item Do our methods allow us to establish polytopality of the supporting cell complexes for the cut ideals of trees of any size?. Do they allow to stablish the symmetry of the $f$-vector of such cellular complexes (which would reflect the arithmetic-Gorenstein nature of cut ideals of trees, which was established in \cite{NP08}).

\item As a last comment, we mention that a computation of the minimal free resolution of the cut ideal of a tree on five vertices yielded a non-unimodal Betti vector $(1, 55, 320, 891, 1436, 1375, 1375, 1436, 891, 320, 55, 1)$. As seen in \cref{tab:diag}, this property washed away when looking at the minimal free resolution of an initial ideal. Is it possible to recover a non-unimodal Betti vector from the initial ideal of cut ideals of trees with respect to some term order? This would amount to a cell complex with a non-unimodal $f$-vector supporting a minimal cellular resolution of the initial ideal. That would be of independent interest in the case that the cell complexes from the previous paragraph are polytopal.

\end{enumerate}

 \subsubsection*{Acknowledgements}
 The authors would like to thank Alexander Engstr\"om for introducing them to the problem and for inspiring discussions and continuous motivation. Also, they would like to thank the Institut Mittag-Leffler for its hospitality, and Alexander Engstr\"om and Bruno Benedetti for organizing the summer school \emph{``Discrete Morse Theory and Commutative Algebra,''} which took place between July 18 and August 2, 2012 at the Institut Mittag-Leffler, and where this project began.

\begin{flushleft}
Samu Potka\\
Aalto University\\
Department of Mathematics and Systems Analysis\\
PO Box 11100\\
FI-00076 Aalto\\
Finland\\
e-mail: samu.potka@aalto.fi
\end{flushleft}

\begin{flushleft}
Camilo Sarmiento\\
Max-Planck-Institute for Mathematics in the Sciences\\
Inselstrasse 22\\
DE-04109 Leipzig\\
Germany\\
e-mail: sarmient@mis.mpg.de
\end{flushleft}

\end{document}

%% file: g33lab.pdf_tex

\begingroup
  \makeatletter
  \providecommand\color[2][]{%
    \errmessage{(Inkscape) Color is used for the text in Inkscape, but the package 'color.sty' is not loaded}
    \renewcommand\color[2][]{}%
  }
  \providecommand\transparent[1]{%
    \errmessage{(Inkscape) Transparency is used (non-zero) for the text in Inkscape, but the package 'transparent.sty' is not loaded}
    \renewcommand\transparent[1]{}%
  }
  \providecommand\rotatebox[2]{#2}
  \ifx\svgwidth\undefined
    \setlength{\unitlength}{99.05146963pt}
  \else
    \setlength{\unitlength}{\svgwidth}
  \fi
  \global\let\svgwidth\undefined
  \makeatother
  \begin{picture}(1,0.64824006)%
    \put(0,0){\includegraphics[width=\unitlength]{g33lab.pdf}}%
    \put(0.79944392,0.01669197){\color[rgb]{0,0,0}\makebox(0,0)[lb]{\smash{$Y$}}}%
    \put(0.3348952,0.56364661){\color[rgb]{0,0,0}\makebox(0,0)[lb]{\smash{$X$}}}%
    \put(-0.00924222,0.01636416){\color[rgb]{0,0,0}\makebox(0,0)[lb]{\smash{$Z$}}}%
  \end{picture}%
\endgroup